\documentclass[12pt]{article}
\usepackage{graphicx,bm,epsf,amsmath,amsthm,amscd,amsfonts,amssymb,amsbsy,color}
\usepackage[hyphens,spaces,obeyspaces]{url}
\usepackage[colorlinks,breaklinks,allcolors=blue]{hyperref}
\usepackage[hdvips]{hypdvips} 
\usepackage{doi}
\theoremstyle{plain}
\newtheorem{theorem}{Theorem}

\newcommand{\e}{\mbox{e}}
\newcommand{\I}{\mbox{Im}}
\newcommand{\Cl}{\mbox{Cl}}
\newcommand{\Li}{\mbox{Li}}

\newcommand{\be}{\begin{equation}}
\newcommand{\ee}{\end{equation}}
\newcommand{\bea}{\begin{eqnarray}}
\newcommand{\eea}{\end{eqnarray}}

\newcommand{\noi}{\noindent}

\newcommand {\MM}[1]{\qquad\mbox{#1}\qquad}
\newcommand {\TF}[6]
{{}_3F_2\!\left(\!\!\begin{array}{c}
#1,\;#2,\;#3\vspace{1mm}\\
#4,\;#5
\end{array}\Big|\,
#6\right)}
\newcommand {\TFF}[8]
{{}_4F_3\!\left(\!\!\begin{array}{c}
#1,\;#2,\;#3,\;#4\vspace{1mm}\\
#5,\;#6,\;#7
\end{array}\Big|\,
#8\right)}
\begin{document}

\title{\bf Two closed-form evaluations for the generalized hypergeometric function
$\bm{{}_4F_3(\frac1{16})}$}
\author{\bf Arjun K. Rathie$^a$ and Mykola A. Shpot$^b$\\
\\
${}^a\!${\em Department of Mathematics,}\\
{\it Vedant College of Engineering and Technology (Rajasthan Technical University),}
\\{\it Tulsi, Bundi District, India}\\[0mm]
{\bf E-Mail: arjunkumarrathie@gmail.com}\\[1mm]
${}^b\!${\it Institute for Condensed Matter Physics, 79011 Lviv, Ukraine}\\[0mm]
{\bf E-Mail: shpot.mykola@gmail.com}}
\date{\today}
\maketitle

\begin{abstract}
\noindent
The objective of this short note is to provide two closed-form evaluations
for the generalized hypergeometric function $_4F_3$ of the argument $\frac1{16}$\,.
This is achieved by means of separating a generalized hypergeometric function $_3F_2$
into even and odd components, together with the use of two known results for
$_3F_2(\pm\frac14)$ available in the literature.
As an application, we obtain an interesting infinite-sum representation for the number $\pi^2$.
Certain connections with the work of Ramanujan and other authors are discussed, involving
other special functions and binomial sums of different kinds.

\end{abstract}

\vskip 1mm
\noindent
\textbf{Keywords.} Generalized hypergeometric functions, Ramanujan-type summation formulas, Clausen function

\vskip 1mm
\noindent
\textbf{2020 Mathematics Subject Classification.} Primary 33C20; Secondary 33C05, 33C90.

\section{Introduction}

The generalized hypergeometric function ${}_pF_q(z)$ with $p$ numerator and $q$ denominator parameters is defined by the series expansion
\begin{equation}\label{A4}
{}_pF_q\Big(\begin{array}{c} a_1,a_2,\ldots,a_p\\
b_1,b_2,\ldots, b_q\end{array}\Big|\,z\Big)
=\sum_{n\geq 0}\dfrac{(a_1)_n(a_2)_n\ldots(a_p)_n}
{(b_1)_n(b_2)_n\ldots(b_q)_n}\;\dfrac{z^n}{n!}\,,
\end{equation}
where $(a)_n$ is the Pochhammer symbol \cite{Pochhammer} (shifted factorial, with $(1)_n=n!$),
\[
(a)_n=\dfrac{\Gamma(a+n)}{\Gamma(a)}=
\begin{cases}a(a+1)\ldots (a+n-1)&n\in\mathbb N;\;\\
             1                   &n=0\,;
\end{cases}\;a\in\mathbb C\,.
\]
As usual, $p$ and $q$ are non-negative integers, and the parameters
$a_j\in\mathbb C, j=1,\ldots ,p$ and $b_k\in\mathbb C\setminus\mathbb Z_0^-,
k=1,\ldots ,q$, where $\mathbb{C}$ denotes the set of complex numbers, $\mathbb Z_0^-$
is the set of negative integers along with zero.

When $p\leq q$, the power series \eqref{A4} converges for all finite values of $z$ with
$|z|<\infty$.
In the case $p=q+1$, it converges for $|z|<1$ and diverges for $|z|>1$.
On the unit circle $|z|=1$,
the sufficient condition for convergence of \eqref{A4} with $p=q+1$
is $\Re(s)>0$ where $s$ is the parametric excess
$$
s:=\sum_{k=1}^q b_k -\sum_{j=1}^p a_j\,.
$$
Also, when $p=q+1$ and $\{|z|=1,\;z\ne1\}$, the series \eqref{A4} converges with
$-1<\Re(s)\leq0$.
Analytic continuation can be employed to define the function ${}_pF_q(z)$ for $z$ values
outside the convergence region of the series \eqref{A4}.%
\footnote{Several interesting instances of the analytic continuation of the $_3F_2$
function are provided in \cite{PS18}.}

For more details on generalized hypergeometric function ${}_pF_q(z)$ we refer to
standard textbooks \cite{Bailey,Slater,Rainville,AAR,NIST}.

The subject of the present short communication are the special representatives of
generalized hypergeometric functions belonging to the class $p=q+1$ with specializations $p=3$ and $p=4$.

The main goal is the derivation of two new closed-form expressions for the function
$_4F_3$ with specific argument $\frac1{16}$ starting from two known evaluations of
$_3F_2(\pm\frac14)$. It is interesting that one of these last $_3F_2$ functions,
$_3F_2(-\frac14)$, has been well-known in mathematics and appeared in the standard
reference book by Prudnikov Brychkov and Marichev \cite{PBM3} as entry 7.4.6.1,
while another one, $_3F_2(\frac14)$, resulted as a by-product in calculations of
certain Feynman integrals by one of the authors in 2010 \cite{Sh10}.

Our motivation consists of several components.
First of all, ($i$) it is of primary interest to obtain explicit determinations
of (any) hypergeometric functions. On the other hand, ($ii$) we wanted to bring
the essentially unknown evaluations \eqref{RP2} and \eqref{5*} from a specific paper in theoretical physics \cite{Sh10} to the broad mathematical community.
We also found it interesting to ($iii$) discuss some implications of the fusion of
summation formulas \eqref{RP1} and \eqref{RP2}, in particular, the infinite-sum
representation of the number $\pi^2$ in \eqref{MMS}.

Finally, it often happens that quite the same mathematical calculations are performed in very different and seemingly disconnected branches of theoretical physics and mathematics.
And it is quite interesting and useful to establish connections between such
differently motivated calculations. Thus, our last motivation's component was
to ($iv$) perform a bright comparison of different approaches, functions and forms of results related to our initial input formulas \eqref{RP1} and \eqref{RP2}.
Thus it appeared that they  were implicitly present in Ramanujan's notebooks, as well as in quite different modern papers in mathematics and theoretical physics.
We hope that our "excursion" from Ramanujan to present days won't be no less interesting for the reader as our main $_4F_3$ results.

\section{Even and odd components of ${}_pF_q(z)$}

By decomposing the generalized hypergeometric function ${}_pF_q(z)$ into even and odd
components and making use of standard identities for the Pochhammer symbol \cite{Pochhammer},
$$
(a)_{2n}=2^{2n}\Big(\frac a2\Big)_n\Big(\frac a2+\frac12\Big)_n\MM{and}
(a)_{2n+1}=a\,2^{2n}\Big(\frac a2+\frac12\Big)_n\Big(\frac a2+1\Big)_n,
$$
one can easily obtain the following two general results, recorded, for example,
in the reference book by Prudnikov et. al. \cite[7.2.3.42]{PBM3}, viz.
\bea\label{BMP}\nonumber
&&{}_{q+1}F_q\Big(\begin{array}{c}a_1,a_2,\ldots ,a_{q+1}\\ b_1,b_2,\ldots , b_q\end{array}
\Big|\,z\Big)+
{}_{q+1}F_q\Big(\begin{array}{c}a_1,a_2,\ldots ,a_{q+1}\\ b_1,b_2,\ldots , b_q\end{array}
\Big|-z\Big)
\\&&
=2\,{}_{2q+2}F_{2q+1}\Big(\begin{array}{c}\frac{a_1}2\,,\frac{a_1}2+\frac12\,,\ldots ,
\frac{a_{q+1}}2,\frac{a_{q+1}}2+\frac12
\\\frac12\,,\frac{b_1}2\,,\frac{b_1}2+\frac12\,,\ldots ,\frac{b_q}2\,,\frac{b_q}2+\frac12\end{array}
\Big|\,z^2\Big)
\eea
and
\bea\label{BMM}\nonumber
&&{}_{q+1}F_q\Big(\begin{array}{c}a_1,a_2,\ldots ,a_{q+1}\\ b_1,b_2,\ldots , b_q\end{array}
\Big|\,z\Big)-
{}_{q+1}F_q\Big(\begin{array}{c}a_1,a_2,\ldots ,a_{q+1}\\ b_1,b_2,\ldots , b_q\end{array}
\Big|-z\Big)
\\&&
=2z\,\frac{a_1a_2\ldots a_{q+1}}{b_1b_2\ldots b_q}\,
{}_{2q+2}F_{2q+1}\Big(\begin{array}{c}\frac{a_1}2+\frac12\,,\frac{a_1}2+1,\ldots ,
\frac{a_{q+1}}2+\frac12\,,\frac{a_{q+1}}2+1
\\\frac32\,,\frac{b_1}2+\frac12\,,\frac{b_1}2+1,\ldots ,\frac{b_q}2+\frac12\,,\frac{b_q}2+1
\end{array}\Big|\,z^2\Big).
\eea

\section{Two closed-form evaluations for ${}_4F_3(\frac1{16})$}

In this section, we employ the above procedure to the two following results:
\be\label{RP1}
\TF{\frac12}{\frac12}{\frac12}{\frac32}{\frac32}{\!\!-\frac14}=\frac{\pi^2}{10}\,,
\ee
well known from \cite[7.4.6.1]{PBM3}, and
\be\label{RP2}
\TF{\frac12}{\frac12}{\frac12}{\frac32}{\frac32}{\frac14}=
\frac{1}{2\sqrt 3}\;\psi^\prime\Big(\frac{1}{3}\Big)-
\frac{\pi^2}{3\sqrt 3}\,,
\ee
obtained in \cite[(26), (41)]{Sh10} in the course of Feynman-graph calculations.
This leads us to two new closed-form summation formulas for the ${}_4F_3$ series with argument $\frac1{16}$, asserted in the following theorem.

\begin{theorem}\label{THM}
The following two closed-form evaluations for the function ${}_4F_3(\frac1{16})$ hold true:
\be\label{T1}
\TFF{\frac14}{\frac14}{\frac14}{\frac34}{\frac12}{\frac54}{\frac54}{\frac1{16}}=
\frac12\left[\pi^2\Big(\frac1{10}-\frac1{3\sqrt3}\Big)+\frac1{2\sqrt3}\,
\psi^\prime\Big(\frac{1}{3}\Big)\right]
\ee
and
\be\label{T2}
\TFF{\frac34}{\frac34}{\frac34}{\frac54}{\frac32}{\frac74}{\frac74}{\frac1{16}}=
36\left[\frac1{2\sqrt3}\,\psi^\prime\Big(\frac{1}{3}\Big)-
\pi^2\Big(\frac1{10}+\frac1{3\sqrt3}\Big)\right],
\ee
where $\psi(z)$ is the psi-function, the logarithmic derivative of the Euler
gamma function $\Gamma(z)$.
\end{theorem}

\begin{proof}
The derivation of the results \eqref{T1} and \eqref{T2}
asserted by the Theorem \ref{THM} are straightforward and
based upon identities \eqref{BMP} and \eqref{BMM}.
We set therein $q=2$ and substitute $a_1=a_2=a_3=\frac12$ and $b_1=b_2=\frac32$, along with
$z=\frac14$.
After some simplifications in resulting hypergeometric $_6F_5$ functions, we make use of
explicit expressions from \eqref{RP1} and \eqref{RP2} and arrive at closed-form
evaluations \eqref{T1} and \eqref{T2}.
This completes the proof of the Theorem \ref{THM}.
\end{proof}

\vspace{4mm}\noi
{\bf Application.}
The following new hypergeometric representation of $\pi^2$ can be easily obtained from
\eqref{T1} and \eqref{T2} by subtracting the appropriately weighted equations, so that
the terms containing the derivative of the $\psi$ function disappear:
\be\label{MMF}
\frac{\pi^2}{10}=
\TFF{\frac14}{\frac14}{\frac14}{\frac34}{\frac12}{\frac54}{\frac54}{\frac1{16}}
-\frac1{72}\,
\TFF{\frac34}{\frac34}{\frac34}{\frac54}{\frac32}{\frac74}{\frac74}{\frac1{16}}.
\ee

Actually, the last formula directly reproduces the original result of \eqref{RP1}
if we recognize here the difference of even and odd parts of
$_3F_2(\frac12,\frac12,\frac12;\frac32,\frac32;\frac14)$.
However, we can alternatively consider the coefficients of $_4F_3$ functions in \eqref{MMF} as
\be
c_1(n)=\frac{\left(\frac{1}{4}\right)_n^3\left(\frac{3}{4}\right)_n}
{\left(\frac{5}{4}\right)_n^2 \left(\frac{1}{2}\right)_n n!}
=
\frac{2^{-2n}\Gamma\left(2n+\frac{1}{2}\right)}{(4n+1)^2\Gamma\left(n+\frac{1}{2}\right)n!}
\ee
and
\be
c_2(n)=\frac{\left(\frac{3}{4}\right)_n^3\left(\frac{5}{4}\right)_n}
{\left(\frac{7}{4}\right)_n^2 \left(\frac{3}{2}\right)_n n!}
=
9\,\frac{2^{-2n}\Gamma\left(2n+\frac{3}{2}\right)}{(4n+3)^2\Gamma\left(n+\frac{3}{2}\right)n!},
\ee
and produce thereof the coefficients of the linear combination in \eqref{MMF} via
\be
c(n)=c_1(n)-\frac1{72}\,c_2(n)=
\frac{2^{-2n}(4n(4n(12n+29)+81)+71)\Gamma\left(2n+\frac{1}{2}\right)}
{16n!(4n+1)^2(4n+3)^2\Gamma\left(n+\frac{3}{2}\right)}\,.
\ee
Now, summing up the whole combination from the right-hand side
of \eqref{MMF} with coefficients $c(n)$, we obtain the following representation of this
identity in the form of the infinite single sum:
\be\label{MMS}
\frac{\pi^2}{10}=
\frac1{16}\sum_{n=0}^\infty\frac{4 n (4 n (12 n+29)+81)+71}{(4 n+1)^2 (4 n+3)^2n!}\,
\frac{\Gamma \left(2 n+\frac{1}{2}\right)}{\Gamma \left(n+\frac{3}{2}\right)}\,
\Big(\frac1{64}\Big)^n.
\ee
We note that the convergence rate of the
sum in \eqref{MMS} is essentially enhanced compared to that in \eqref{RP1}:
As $n\to\infty$, the terms in the last sum decay $\propto4^{-n}$ times faster as in
the original $_3F_2$ function in \eqref{RP1}.

\vspace{1mm}
The results \eqref{T1} through \eqref{MMS} have been verified
using the MATHEMATICA software \cite{Math12}.

\vspace{4mm}\noi
{\bf Remark 1.} The method of splitting the generalized hypergeometric series \eqref{A4}
into its even and odd parts to generate new evaluations of higher-order hypergeometric functions is not new.
It goes back at least to Krupnikov and K\"olbig \cite[p. 84]{KK97} and Exton \cite{Exton97},
and has been used recently in a number of publications, among them \cite{LR22,KKR22,KLR22}.
Other aspects of separation a power series into its even and odd components are
considered in \cite{HMS79} and \cite[Chap. 3]{SriMan}.

\section{Relation to Ramanujan's notebooks, other functions, binomial and harmonic sums}

{\bf Remark 2.} Due to the relation \cite[(15)]{Gj84}, \cite[p. 329]{dD84}, \cite[(1.10)]{CS14}
\be\label{CP}
\Cl_2\Big(\frac{\pi}3\Big)=
\frac{1}{2\sqrt 3}\left[\psi^\prime\Big(\frac{1}{3}\Big)-\frac23\,\pi^2\right],
\ee
the results \eqref{RP2}, \eqref{T1} and \eqref{T2} can be alternatively expressed in terms of
the maximum value $\Cl_2(\frac{\pi}3)$ of the Clausen function $\Cl_2(\theta)$
\cite[Ch. 4]{Lewin}, \cite{CS14}
\be\label{DCL}
\Cl_2(\theta)=-\int_0^\theta d\theta\ln\Big|2\sin\frac{\theta}{2}\Big|
=\sum_{n\ge1}\frac{\sin n\theta}{n^2}=\I\Li_2(\e^{i\theta})\,.
\ee
This function, along with its natural extensions, is very important in the theory of dilogarithm, $\Li_2(z)$, and polylogarithm functions \cite{Lewin}, and in computations of the Feynman integrals
\cite{DavKal01,Weinzierl}.

\vspace{4mm}\noi
{\bf Remark 3.} A few results, equivalent to the one in \eqref{RP2}, have appeared
in the literature before \cite[(26), (41)]{Sh10}, though in quite different forms.

The first one goes back to Ramanujan, and appears in Example $(iii)$ of Entry 16 in
\cite[p. 40]{BJ-R} and \cite[p. 264]{Berndt1}, as an evaluation of the binomial sum
\be\label{BS}
\sum_{k\ge0}\Big({2k\atop k}\Big)\frac1{2^{4k+1}(2k+1)^2}=
\frac{3\sqrt3}4\sum_{k\ge0}\frac1{(3k+1)^2}-\frac{\pi^2}{6\sqrt3}
\ee
where
\be\label{BC}
2^{-2k}\Big({2k\atop k}\Big)=\frac1{k!}\Big(\frac12\Big)_k=2^{-2k}\,\frac{(2k)!}{(k!)^2}
=\frac{(2k-1)!!}{(2k)!!}:=\mu_k
\ee
is the normalized binomial mid-coefficient (see e. g. \cite[(2.2)]{AKS06} and related references). Indeed, using for the function
$_3F_2(\frac12,\frac12,\frac12;\frac32,\frac32;\frac14)$ from \eqref{RP2} its series definition \eqref{A4} along with the first equality from \eqref{BC} and the simple relation
$(\frac12)_k/(\frac32)_k=(2k+1)^{-1}$ between Pochhammer symbols, we see that it matches the
binomial sum on the left-hand side of \eqref{BS} up to an extra factor of $\frac12$ in \eqref{BS}. To match the right-hand sides of \eqref{RP2} and \eqref{BS}, we have to take into account that \cite[5.15.1]{NIST}, \cite[5.1.7.16]{PBM1}
$$
\sum_{k\ge0}\frac1{(3k+1)^2}=\frac19\,\psi^\prime\Big(\frac{1}{3}\Big).
$$

The same result has been reproduced by Zucker \cite[(2.13)]{Z85} in the form
\be\label{BZ}
\sum_{k\ge0}\Big({2k\atop k}\Big)\frac1{4^{2k}(2k+1)^2}=
\frac{3\sqrt3}4\,L_{-3}(2).
\ee
In notations of Zucker, $\frac{3\sqrt3}4\,L_{-3}(2)$ means the imaginary part of dilogarithm $\Li_2(\e^{i\pi/3})$ via \cite[(1.11)]{Z85} and \cite[(1.8)]{Z85}.
Hence we see, by noticing \eqref{DCL} and \eqref{CP},
that his equation \eqref{BZ} is tantamount just to \eqref{RP2} expressed as
\be\label{5*}\tag{$5^*$}
\TF{\frac12}{\frac12}{\frac12}{\frac32}{\frac32}{\frac14}=\Cl_2\Big(\frac{\pi}3\Big).
\ee

As noticed after \eqref{RP2}, this relation appeared in \cite{Sh10} as a result of a calculation of a Feynman integral. A slightly different treatment of the same integral in
\cite{PV00np} yielded an alternative relation,
\be\label{PV}
\frac{3\sqrt\pi}8\sum_{n=0}^\infty\frac{\Gamma(n+1)\psi(n+1)}{\Gamma(n+\frac32)}\,4^{-n}+
\frac{\pi\sqrt3}6\left(\gamma_E+\ln3\right)=
\frac2{\sqrt3}\,\Cl_2\Big(\frac{\pi}3\Big)=\frac13\,\psi^\prime\Big(\frac{1}{3}\Big)-
\frac{2\pi^2}9\,,
\ee
which is obtained by combining (A13) and (A.15) of \cite{PV00np}%
\footnote{The reference to \cite{PV00np} given in \cite{Sh10} is incorrect:
Ref. [13] in \cite{Sh10} has to be replaced by \cite{PV00np}.}.
In \eqref{PV}, $\gamma_E$ is the Euler constant, given by the limiting behavior of
harmonic numbers%
\footnote{For more related information, see e. g. \cite[p. 147]{AKS06}.}
\be\label{Hn}
H_n:=\sum_{k=1}^n\frac1k=\gamma_E+\psi(n+1)
\MM{via}
\gamma_E=\lim_{n\to\infty}\left(H_n-\ln n\right).
\ee
Using the first equation in \eqref{Hn} and \cite[Annex 1.6]{PBM1}
$$
\Gamma\Big(n+\frac32\Big)=\frac{\sqrt\pi}2\,2^{-n}\,\frac{(2n+1)!}{n!},
$$
we can rewrite \eqref{PV} in a more instructive form:
\be\label{IB}
\sum_{n=0}^\infty\frac{H_n}{\left({2n\atop n}\right)(2n+1)}=
\frac8{3\sqrt3}\,\Cl_2\Big(\frac{\pi}3\Big)-\frac{2\pi}{3\sqrt3}\,\ln3.
\ee
Now, combining \eqref{BS}, \eqref{BZ}, \eqref{5*}, and \eqref{IB}, we obtain a nice chain of equalities,
\be\label{R1}
\TF{\frac12}{\frac12}{\frac12}{\frac32}{\frac32}{\frac14}=
\sum_{k=0}^\infty\Big({2k\atop k}\Big)\frac{4^{-2k}}{(2k+1)^2}=
\frac{3\sqrt3}8\,\sum_{k=0}^\infty\frac{H_k}{\left({2k\atop k}\right)(2k+1)}+\frac{\pi}4\,\ln3=
\Cl_2\Big(\frac{\pi}3\Big),
\ee
relating the generalized hypergeometric function from \eqref{RP2} and \eqref{5*} with corresponding binomial and inverse binomial harmonic sums.

It seems that identifications of results \eqref{RP2} and \eqref{5*} in terms of
a generalized hypergeometric function $_3F_2$ of argument $\frac14$ appear only in equations
(26) and (41) of reference \cite{Sh10}.

\vspace{4mm}\noi
{\bf Remark 4.} Finally, we mention a similar connection of the $_3F_2$ function of argument $-\frac14$ from \eqref{RP1}
with the sum $f(x)$ considered by Ramanujan in Example $(ii)$ of Entry 8 (see
\cite[p. 24]{BJ-R}, \cite[p. 250]{Berndt1}), and that from \cite[(3.9)]{Z85}.

Similarly as before, we easily establish the equivalence of \eqref{RP1} with \cite[(3.9)]{Z85}, namely
\be\label{BZ9}
\sum_{k\ge0}\Big({2k\atop k}\Big)\frac{(-1)^k}{2^{4k}(2k+1)^2}=\frac{\pi^2}{10}.
\ee
Moreover, Zucker \cite[p. 101]{Z85} notices the equivalence of his equation (3.9) with the special result $f(1/\sqrt5)=\pi^2/20$ for the Ramanujan's sum
\be\label{RS}
f(x)=\sum_{k=1}^\infty\frac{h_k\,x^{2k-1}}{2k-1},\MM{where} h_k=\sum_{j=1}^k\frac1{2j-1}.
\ee
Hence, combining \eqref{RP1}, \eqref{BZ9}, and \eqref{RS}, we may write
\be\label{R2}
\TF{\frac12}{\frac12}{\frac12}{\frac32}{\frac32}{-\frac14}=
\sum_{k=0}^\infty\Big({2k\atop k}\Big)\frac{(-16)^{-k}}{(2k+1)^2}=
2\sum_{k=1}^\infty\frac{h_k\,5^{-k+\frac12}}{2k-1}=
\frac{\pi^2}{10}.
\ee
The relations of the odd harmonic number $h_k$ to harmonic numbers $H$ and digamma
function $\psi$ are given by \cite[(2.3)]{AKS06}, viz.
$$
h_k=H_{2k}-\frac12\,H_k=\frac12\left[\psi\Big(k+\frac12\Big)-\psi\Big(\frac12\Big)\right].
$$

We are quite sure that the relations \eqref{R1} and \eqref{R2} deserve further analytical investigation,
in particular towards generalizations to other arguments of $_3F_2$ functions, apart from $\pm\frac14$.

\section{Concluding remarks}
In this note, we have provided two closed-form evaluations for the generalized hypergeometric function $_4F_3$ of the argument $\frac1{16}$\,.
A combination of these results yields an interesting representation \eqref{MMS}
of the transcendental number $\pi^2$ in terms of an infinite sum with an enhanced convergence rate.

There are many theoretical and practical applications of $_pF_q$ functions in
different fields of mathematics \cite{Milg18,FGR20,RP21,RMP22,KP22}, statistics and combinatorics \cite{RP21,RKP22,KKR22}, theoretical physics \cite{SS15,PS18,Anti} and engineering \cite{SRS18,RGK21}.%
\footnote{It is quite impossible to give an exhaustive list
of relevant references here; our choice is motivated by scientific interests of the authors.}
Hopefully, our results could be potentially useful in some of these areas.
To the best of our knowledge, the established summation formulas \eqref{T1} and \eqref{T2}
did not appear in the literature before; we believe, they represent a definite contribution to the theory of generalized hypergeometric functions.

Moreover, we have established certain connections of functions under consideration with the work
of Ramanujan and other authors, with other special functions, and with binomial sums of different kinds, some of them including harmonic numbers $H_k$ and $h_k$.

A large number of results closely related to that of
\eqref{T1}, \eqref{T2} and \eqref{MMS} are under investigation and will be subject of a subsequent
publication.

\providecommand{\href}[2]{#2}

\end{document}